\documentclass[12pt]{amsart}
\usepackage{amssymb}
\usepackage{graphicx,color}
\usepackage{amscd}
\usepackage{stmaryrd}
\usepackage{epsfig}
\usepackage{multirow}

\calclayout
\allowdisplaybreaks

\newcommand\grad{\operatorname{grad}}
\renewcommand\div{\operatorname{div}}
\newcommand\curl{\operatorname{curl}}

\renewcommand\sp{\operatorname{span}}

\newcommand\R{\mathbb{R}}

\newcommand\eps{\operatorname\epsilon}
\newcommand\x{\times}
\newcommand\A{{\mathcal A}}
\newcommand\B{{\mathcal B}}

\renewcommand\P{{\mathcal P}}

\newcommand\LL{{\mathcal L}}

\newcommand\T{{\mathcal T}}

\newcommand\supp{\operatorname{supp}}
\renewcommand{\Pi}{\varPi}

\newcommand{\<}{\langle}
\renewcommand{\>}{\rangle}

\long\def\COMMENT#1{\par\vbox{\hrule\vskip1pt \hrule width.125in height0ex #1\vskip1pt\hrule}}

\numberwithin{equation}{section}

\newtheorem{thm}{Theorem}[section]

\newtheorem{lem}[thm]{Lemma}

\begin{document}

\long\def\COMMENT#1{\par\vbox{\hrule\vskip1pt#1\vskip1pt\hrule}}

\title[A uniform inf--sup condition]{A uniform inf--sup condition with 
applications to preconditioning}
\author{Kent--Andre Mardal}
\address{Centre for Biomedical Computing at Simula Research Laboratory, Norway and
Department of Informatics, University of Oslo, Norway}
\email{kent-and@simula.no}
\urladdr{http://simula.no/people/kent-and}
\thanks{The first author was supported by the Research Council of Norway
through Grant 209951 and a Centre of Excellence grant
to the Centre for Biomedical Computing at Simula Research Laboratory.}
\author{Joachim Sch\"{o}berl}
\address{Institute for Analysis and Scientific Computing, 
Wiedner Hauptstrasse 8-10, 
1040 Wien, Austria}
\email{joachim.schoeberl@tuwien.ac.at}
\urladdr{http://www.asc.tuwien.ac.at/~schoeberl}
\thanks{}
\author{Ragnar Winther}
\address{Centre of Mathematics for Applications
and Department of Informatics,
University of Oslo, 0316 Oslo, Norway}
\email{ragnar.winther@cma.uio.no}
\urladdr{http://folk.uio.no/rwinther}
\thanks{The third author was supported by the Research Council of Norway
through a Centre of Excellence grant
to the Centre of Mathematics for Applications}
\date{October 31, 2011}

\begin{abstract}
A uniform inf--sup condition related to a parameter dependent Stokes problem is established.
Such conditions are intimately connected to the construction of
uniform 
preconditioners for the problem, i.e., 
preconditioners which behave uniformly well with respect to variations 
in the model parameter as well as the 
discretization parameter. For the present model, similar results 
have been derived before, but only by
utilizing extra regularity ensured by convexity of the domain. The
purpose of this paper is to remove 
this artificial assumption. As a byproduct of our analysis, 
in the two dimensional case we also construct a new projection operator for
the Taylor--Hood element  
which is uniformly bounded in $L^2$ and commutes with the divergence
operator. 
This construction is based on a tight connection between a subspace of
the Taylor--Hood velocity space and the lowest order Nedelec edge element.
\end{abstract}

\maketitle

\section{Introduction}\label{sec:intro}
The purpose of this paper is to discuss preconditioners for finite
element discretizations of a singular 
perturbation problem related to the linear Stokes problem.
More precisely, 
let $\Omega \subset \R^n$ be a bounded Lipschitz domain and $\eps \in
(0,1]$ a real parameter. We will consider singular
perturbation
problems of the form
\begin{equation}\label{stokes1}
\begin{array}{rl}
(I - \eps^2 \Delta) u - \grad p &= f \quad \text{in } \Omega,\\
\div u &= g \quad \text{in } \Omega,\\
u &= 0 \quad \text{on } \partial \Omega,
\end{array}
\end{equation}
where the unknowns $u$ and $p$ are a vector field and a scalar field,
respectively. For each fixed positive  value of the perturbation parameter
$\eps$ the problem behaves like the Stokes system,
but formally the system approaches a so--called mixed formulation of 
a scalar Laplace equation as this parameter tends to zero.
In physical terms this means that we are studying 
fluid flow in regimes ranging from linear Stokes flow to 
porous medium flow. Another motivation for studying preconditioners 
of these systems 
is that they frequently arises as subsystems in 
time stepping schemes for time dependent Stokes and Navier--Stokes
systems, cf. for example \cite{B-P-97, C-C-88,
  M-W-04,StokesInterface3, turek}.

The phrase {\it uniform preconditioners} for parameter
dependent problems like the ones we discuss here, 
refers to the ambition to construct preconditioners such that the 
preconditioned systems have 
condition numbers which are bounded
uniformly with respect to the perturbation parameter $\eps$ and 
the discretization. Such results have been obtained for the
system\eqref{stokes1}
in several of the
studies mentioned above, but a necessary assumption in all the studies
so far has been 
a {\it convexity assumption} on the domain, cf. \cite{errata}. 
However, below in Section~\ref{uniform}
we will present a numerical example which clearly indicates 
that this assumption should not be necessary. Thereafter, we will give a
theoretical justification for this claim. The basic tool for achieving 
this is to introduce the Bogovski\u{i} operator, cf. \cite{cos-mc}, as the
proper right inverse of the divergence operator in the continuous case.

The construction of uniform preconditioners for discretizations of systems of the form
\eqref{stokes1},
is intimately connection to the
well--posedness
properties of the continuous system, and the stability of the
discretization.
In fact, if we obtain appropriate $\eps$--independent bounds on the
solution operator, then the basic structure of a uniform
preconditioner
for the continuous system is an immediate consequence. Furthermore,
under the assumption of proper stability properties of the
discretizations,
the basic structure of uniform preconditioners for the discrete system
also follows. We refer to \cite{review} and references given there for 
a discussion of these issues.
The main tool for analyzing
the well--posedness properties of saddle--point problems of the form 
\eqref{stokes1} is the Brezzi conditions, cf. \cite{Brezzi,BrezziFortin}.
In particular, the desired
uniform bounds on the solution operator 
is closely tied to a  {\it uniform inf--sup condition} of the form 
\eqref{u-inf-sup} stated  below. Furthermore, the verification of such
uniform conditions are closely tied to the construction of uniformly
bounded projection operators which properly commute with the divergence
operator. In the present case, these projection operators have  to be
bounded both in 
$L^2$ and in $H^1$. In Section~\ref{discrete-problems} we will
construct such operators in the case of the Mini element and 
the Taylor--Hood element, where the latter construction is restricted
to quasi--uniform meshes in two space dimensions.

\section{Preliminaries}\label{prelim}
To state the proper uniform inf--sup condition for the system
\eqref{stokes1}
 we will need some
notation. 
If $X$ is a Hilbert space, then $\|\cdot \|_X$ denotes its norm.
We will use $H^m= H^m(\Omega)$ to denote the Sobolev space of 
functions on $\Omega$ with $m$ derivatives in $L^2 = L^2(\Omega)$. The
corresponding spaces for vector fields are denoted $H^m(\Omega;\R^n)$
and
$L^2(\Omega;\R^n)$. 
Furthermore, $\< \cdot , \cdot
\>$ is used to denote the inner--products in both $L^2(\Omega)$ and
$L^2(\Omega;\R^n)$, 
and it will also denote various duality pairings obtained by 
extending these inner--products. 
In general, we will use 
$H^m_0$ to denote  
the closure in $H^m$ of the space of smooth functions with compact
support in $\Omega$, and the dual 
space of $H^m_0$
with respect to the $L^2$ inner product by $H^{-m}$.
Furthermore, $L^2_0$ will denote the space of $L^2$ functions with mean 
value zero. We will use $\LL(X,Y)$ to denote the space of bounded
linear operators mapping elements of $X$ to $Y$, and if $Y= X$ we
simply write $\LL(X)$ instead of $\LL(X,X)$.

If $X$ and $Y$ are Hilbert spaces, both continuously contained in some
larger Hilbert space, then the intersection $X \cap Y$ and  the sum
$X+ Y$
are both Hilbert spaces with norms given by
\[
\|x \|_{X \cap Y}^2 = \|x \|_{X}^2 + \|x \|_{Y}^2 \quad \text{and }  
\|z \|_{X+Y}^2 = \inf_{\substack{x\in X, y \in Y\\z = x+y}} 
(\|x \|_{X}^2 + \|y \|_{Y}^2 ).
\]
Furthermore, if $X \cap Y$ are dense in both the Hilbert spaces 
$X$ and $Y$ then $(X \cap Y)^*= X^* + Y^*$ and $(X + Y)^*= X^* \cap
Y^*$, cf. \cite{b-l}.

The system \eqref{stokes1} admits the following weak formulation:

Find $(u,p) \in H_0^1(\Omega;\R^n) \x L^2_0(\Omega)$ such that
\begin{equation}\label{stokes2}
\begin{array}{rll}
\<u , v \> + \eps^2 \<Du, Dv \> &+\, \<p,\div v \> &= \< f,v\>, \quad v \in
H_0^1(\Omega;\R^n),\\
\<\div u, q \> &   &= \< g, q \>, \quad q \in 
L^2_0(\Omega)
\end{array}
\end{equation}
for given data $f$ and $g$. Here $Dv$ denotes the gradient of the 
vector field $v$. More compactly, we can write this system in the form
\begin{equation}\label{stokes3}
\A_{\eps}\begin{pmatrix} u\\ p \end{pmatrix} = \begin{pmatrix} f\\
  g \end{pmatrix}, \quad \text{where }
\A_{\eps} = \begin{pmatrix} I - \eps^2 \Delta & -\grad \\ \div & 0 \end{pmatrix}.
\end{equation}
For each fixed positive $\eps$ the coefficient operator $\A_{\eps}$ is an
isomorphism mapping $X= H_0^1(\Omega;\R^n) \x L^2_0(\Omega)$
onto $X^*= H^{-1}(\Omega;\R^n) \x L^2_0(\Omega)$. However, the operator
norm $\| \A_{\eps}^{-1} \|_{\LL(X^*,X)}$ will blow up as $\eps$ tends
to zero. 

To obtain a proper uniform bound on the operator norm for the solution
operator
$\A_{\eps}^{-1}$ we are forced to introduce $\eps$ dependent spaces
and norms. We define the spaces $X_{\eps}$ and $X_{\eps}^*$ by
\[
X_{\eps} = (L^2 \cap \eps H^1_0)(\Omega;\R^n) \x ((H^1\cap L^2_0) +
\eps^{-1}L_0^2)(\Omega) 
\]
and 
\[
X_{\eps}^* = (L^2 + \eps^{-1} H^{-1})(\Omega;\R^n) \x (H^{-1}_0 \cap \eps L^2_0)
(\Omega).
\]
Here $H^{-1}_0 \supset L^2_0$ corresponds to the  dual space of $H^1
\cap L^2_0$.
Note that the space $X_{\eps}$ is equal to $X$ as a set, but the
norm approaches the $L^2$--norm as $\eps$ tends to zero. 

Our strategy is to use the Brezzi conditions \cite{Brezzi,BrezziFortin} to claim that 
the operator norms
\begin{equation}\label{mapping-property} 
\| \A_{\eps}\|_{\LL(X_{\eps},X_{\eps}^*)} \quad \text{and } 
\| \A_{\eps}^{-1}\|_{\LL(X_{\eps}^*,X_{\eps})} \quad \text{are bounded independently
of } \eps. 
\end{equation}
In fact, the only nontrivial condition 
for obtaining this is that we need to verify the
uniform inf--sup condition 
\begin{equation}\label{u-inf-sup}
\sup_{v \in H_0^1(\Omega;\R^n)}\frac{\<\div v, q\>}{\| v \|_{L^2\cap
    \eps H^1}} \ge \alpha
\| q \|_{H^1 + \eps^{-1}L^2}, \quad q \in L_0^2(\Omega),
\end{equation}
where the positive constant $\alpha$ is independent of $\eps \in
(0,1]$. Of course, if $\eps > 0$ is fixed, and $\alpha$ is allowed to 
depend on $\eps$, then this is just 
equivalent to the standard inf--sup condition for the stationary
Stokes problem.

As explained, for example in \cite{review}, the mapping property
\eqref{mapping-property}
implies that the ``Riesz operator'' $\B_{\eps}$, mapping $X_{\eps}^*$
isometrically to $X_{\eps}$, is a uniform preconditioner for the operator
$\A_{\eps}$. More precisely, up to equivalence of norms the operator
$\B_{\eps}$ can be identified as 
the block diagonal and
positive definite operator $\B_{\eps} : X_{\eps}^* \to X_{\eps}$, given 
by 
\begin{equation}\label{u-preconditioner}
\B_{\eps} = \begin{pmatrix} (I - \eps^2 \Delta)^{-1} & \\ 0  &
  (-\Delta)^{-1}+ \eps^2 I \end{pmatrix}.
\end{equation}
This means that the preconditioned coefficient operator
$\B_{\eps}\A_{\eps}$ is a uniformly bounded family of operators on the
spaces $X_{\eps}$, with uniformly bounded inverses. Therefore, the
preconditioned system
\[
\B_{\eps}\A_{\eps} \begin{pmatrix} u \\p \end{pmatrix} = \B_{\eps}
\begin{pmatrix} f \\ g \end{pmatrix}
\]
can, in theory, be solved by a standard iterative method like a Krylov space
method, with a uniformly bounded convergence rate.  
We refer to \cite{review} for more details.
Of course, for practical computations we are really interested in 
the corresponding discrete problems. This will be further discussed 
in Section \ref{discrete-problems} below.

\section{The uniform inf--sup condition}\label{uniform}
The rest of this paper is devoted to verification of the uniform inf--sup condition
\eqref{u-inf-sup}, and its proper discrete analogs.
We start this discussion by considering the standard stationary
Stokes problem given by:

Find $(u,p) \in H_0^1(\Omega;\R^n) \x L^2_0(\Omega)$ such that
\begin{equation}\label{stokes}
\begin{array}{rll}
\<Du, Dv \> &+\, \<p,\div v \> &= \< f,v\>, \quad v \in
H_0^1(\Omega;\R^n),\\
\<\div u, q \> &   &= \< g, q \>, \quad q \in 
L^2_0(\Omega),
\end{array}
\end{equation}
where $(f,g) \in H^{-1}(\Omega; \R^n) \x L^2_0(\Omega)$. 
The unique solution
of this problem satisfies the estimate
\begin{equation}\label{stokes-bound}
\| u \|_{H^1} + \|p \|_{L^2} \le c \, (\|f \|_{H^{-1}} + \| g \|_{L^2}),
\end{equation}
cf. \cite{g-r}.
Furthermore, if the domain $\Omega$ is convex, $f \in
L^2(\Omega;\R^n)$, and $g = 0$ 
then $u \in H^2\cap H_0^1(\Omega;\R^n)$, $p \in H^1\cap
L^2_0(\Omega)$,
and 
an improved
estimate of the form
\begin{equation}\label{stokes-bound-improved}
\| u \|_{H^2} + \|p \|_{H^1} \le c \, \|f \|_{L^2},
\end{equation}
holds (\cite{Dauge}). 

We define  $R \in \LL(H^{-1}(\Omega, \R^n),L^2_0(\Omega))$ to be the 
solution of operator of the system \eqref{stokes}, with $g = 0$,  given by $f \mapsto
p = Rf$. 
Hence, if the domain $\Omega$ is convex,
this operator will also be a bounded map of $L^2(\Omega;\R^n)$ into $H^1\cap
L^2_0(\Omega)$.
Furthermore, let $S \in \LL(L^2_0(\Omega),H_0^1(\Omega;\R^n)$ denote 
the corresponding solution operator, defined by \eqref{stokes} with $f
= 0$,  given by $g \mapsto u = Sg$. Then $S$ is a right inverse of the
divergence operator, and the operator 
$S$ is the adjoint of $R$, since
\[
\< f, Sg \> = \<Rf, \div Sg \> + \< D u, D Sg \> = \< Rf, g \> + \<
p,\div u\> = \< Rf, g \>.
\]
Here $u$ and $p$ are components of the solutions of \eqref{stokes}
with data $(f,0)$ and $(0,g)$, respectively. 
As a consequence of the improved estimate
\eqref{stokes-bound-improved}, 
we can therefore conclude that if the domain $\Omega$ is convex then $S$ can be extended to 
an operator in $\LL( H_0^{-1}(\Omega), L^2(\Omega;\R^n))$.
In other words, in the convex case we have
\begin{equation}\label{bounded-right-inverse}
S \in \LL(L_0^2,H_0^1)\cap
\LL(H_0^{-1},L^2), \quad 
\text{and } \, \div Sg = g. 
\end{equation}
However, the existence of such a right inverse of the divergence
operator implies that the
uniform
inf--sup condition holds, since for any $q \in L_0^2(\Omega)$ we have
\[
\|q \|_{H^1 + \eps^{-1}L^2} = \sup_{g \in H_0^{-1}\cap \eps L^2}
\frac{\< g, q \>}{\| g \|_{H_0^{-1}\cap \eps L^2}}
\le c \sup_{g \in H_0^{-1}\cap \eps L^2}
\frac{\< \div S g, q \>}{\| Sg \|_{L^2\cap \eps H^1}}
\le c \sup_{v \in L^2\cap \eps H_0^1}
\frac{\< \div v, q \>}{\| v \|_{L^2\cap \eps H^1}}.
\]
On the other hand, if the domain $\Omega$ is not convex, 
then the estimate \eqref{stokes-bound-improved}
is not valid, and as a consequence, the operator $S$
cannot be extended to an operator in $\LL( H_0^{-1}(\Omega),
L^2(\Omega;\R^n))$. Therefore, the proof of the 
uniform inf--sup condition 
\eqref{u-inf-sup} outlined above breaks down in the nonconvex case.

\subsection{General Lipschitz domains}\label{general-domain}
The main purpose of this paper is to show that the problems 
encountered above for nonconvex domains are just technical problems
which can be overcome. As a consequence, preconditioners
of the form $\B_{\eps}$ given by \eqref{u-preconditioner}
will still behave as a uniform preconditioner in the nonconvex case.
To convince the reader that this is indeed a reasonable hypothesis we
will first present a numerical experiment. We consider the problem 
\eqref{stokes1} on three two dimensional domains, referred to as 
$\Omega_1$, $\Omega_2$ and $\Omega_3$. Here $\Omega_1$ is the unit
square, $\Omega_2$ is the L--shaped domain obtained by cutting out the
an upper right subsquare from $\Omega_1$, while $\Omega_3$ is the slit
domain where a slit of length a half is removed from from
$\Omega_1$, cf. Figure~\ref{domains}. Hence, only $\Omega_1$ is a convex domain.
\begin{figure}[htb]
\centerline{\includegraphics[width=2.5in]{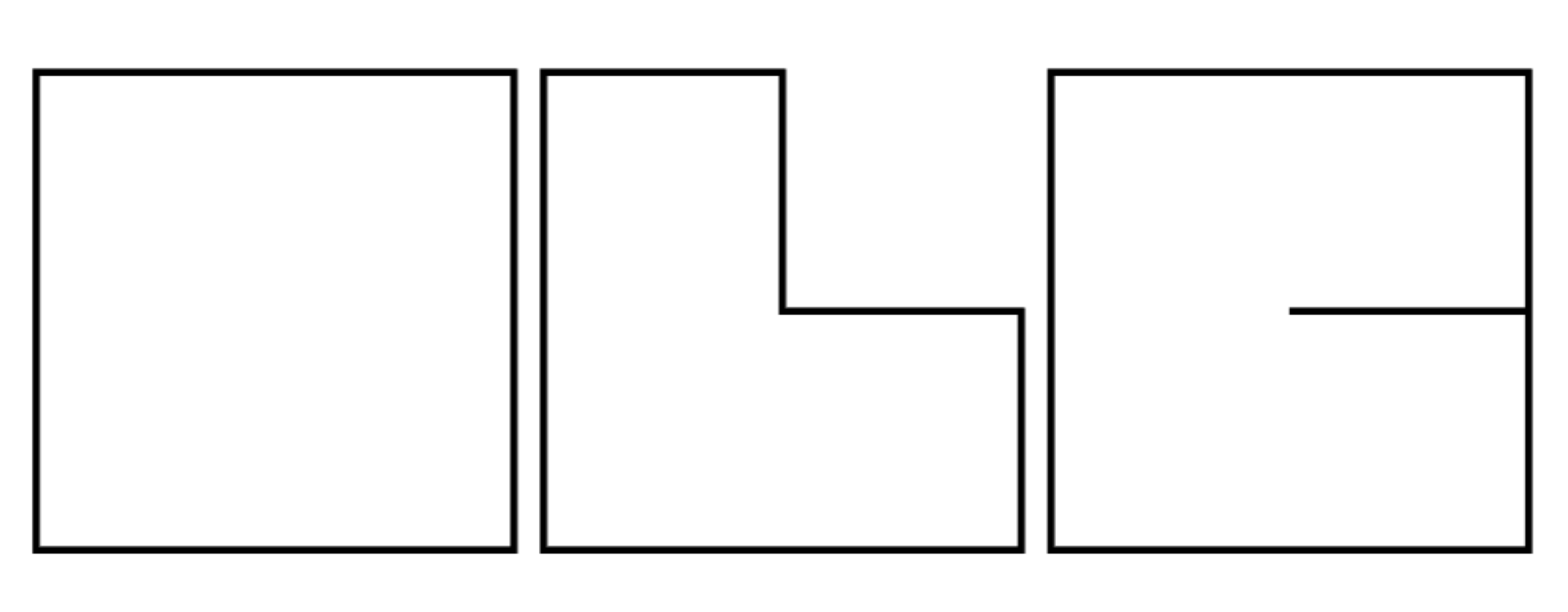}}
\caption{The domains $\Omega_1$, $\Omega_2$,
and $\Omega_3$.}
\label{domains}
\end{figure}
The corresponding problems \eqref{stokes1} were discretized by the
standard Taylor--Hood element on a uniform triangular grid to obtain a 
discrete anolog of this system \eqref{stokes3} on the form
\[
\A_{\eps,h}\begin{pmatrix} u_h \\ p_h \end{pmatrix}
= \begin{pmatrix} f_h \\ g_h \end{pmatrix}.
\]
Here the parameter $h$ indicates the mesh size.
We have computed the condition numbers of the operator 
$\B_{\eps,h}\A_{\eps,h}$ for different values of $\eps$ and $h$
for the three domains. The operator $\B_{\eps,h}$ is given as the
corresponding
discrete version of \eqref{u-preconditioner}, i.e., exact inverses of
the discrete elliptic operators appearing in \eqref{u-preconditioner}
are used. Hence, in the notation of \cite{review} a canonical
preconditioner is applied.
The results are given in Table~\ref{table1} below.

\begin{table}
\begin{center}
\begin{tabular}{|c|c||c|c|c|c|} \hline 
domain  & $\eps \backslash h$  &   $2^{-2}$  & $2^{-3}$ & $2^{-4}$ & $2^{-5}$ \\  \hline  
\multirow{3}{20mm}{$\quad\Omega_1$} 
      & 1 & 14.3 & 15.3 & 18.5 & 22.0 \\ 
      & 0.1 & 10.3& 11.9 & 12.7& 13.2 \\ 
      & 0.01 & 6.2 & 6.7 & 8.0 & 9.9 \\  \hline
\multirow{3}{20mm}{$\quad\Omega_2$} 
      & 1 & 17.1 & 17.2 & 17.1& 17.1 \\ 
      & 0.1 & 10.3 & 11.8 & 12.7 & 13.2  \\ 
      & 0.01 & 6.1 & 6.7 & 8.0 & 9.9 \\ \hline 
\multirow{3}{20mm}{$\quad\Omega_3$} 
      & 1 & 13.2 & 13.4 & 13.5 & 13.6 \\ 
      & 0.1 & 10.3 & 11.9 & 12.8 & 13.2  \\ 
      & 0.01 & 6.1 & 6.7 & 8.1 & 9.9 \\ \hline 
\end{tabular}
\caption{Condition numbers for the operators $\B_{\eps,h}\A_{\eps,h}$} 
\label{table1}
\end{center}
\end{table}
These results indicate clearly that the condition numbers 
of the operators $\B_{\eps,h}\A_{\eps,h}$
are not dramatically effected by lack of convexity of the domains.
Actually, we will show below that these condition numbers are indeed 
uniformly bounded both with respect to the perturbation parameter $\eps$
and the discretization parameter $h$.

We will now return to a verification of the 
inf--sup condition \eqref{u-inf-sup} for general Lipschitz domains.
The problem we encountered above in the nonconvex case is caused by
the lack of regularity of the solution operator for Stokes problem
on general Lipschitz domains. However, to establish \eqref{u-inf-sup}
we are not restricted to such solution operators. 
It should be clear from the discussion above that if 
we can find {\it any} operator $S$ satisfying condition
\eqref{bounded-right-inverse},
then \eqref{u-inf-sup} will hold. 
A proper operator which satisfy these conditions is  
the Bogovski\u{i} operator, see for example \cite[section
III.3]{galdi}, or \cite{cos-mc,g-h-h}.
On a domain $\Omega$, which is star shaped with respect to an open
ball $B$,  this operator is explicitly given as
an integral operator on 
the form
\[
Sg(x) = \int_{\Omega}g(y)K(x-y,y)\, dy, \quad \text{where } K(z,y) =
\frac{z}{|z|^n}\int_{|z|}^\infty \theta(y +r\frac{z}{|z|})r^{n-1} \, dr.
\]
Here $\theta \in C_0^\infty(\R^n)$ with 
\[
\supp \theta \subset B, \quad \text{and } \int_{\R^n} \theta(x) \, dx
= 1.
\]
This operator is a right inverse of the divergence operator, and it
has exactly  the desired
mapping properties given by \eqref{bounded-right-inverse}, cf. 
\cite{cos-mc,g-h-h}.
Furthermore, the definition of the right inverse $S$ can also be extended to general
bounded Lipschitz domains, by using the fact that such domains 
can be written as a finite union of star shaped domains. 
The constructed operator will again satisfy the properties given by
\eqref{bounded-right-inverse}.
We refer to 
\cite[section
III.3]{galdi}, \cite[section 2]{g-h-h}, and \cite[section 4.3]{cos-mc}
for more details. We can therefore conclude our discussion so far with the
following theorem.

\begin{thm}\label{main-cont}
Assume that $\Omega$ is a bounded Lipshitz domain. Then the uniform
inf--sup condition \eqref{u-inf-sup} holds.
\end{thm}

\section{Preconditioning the discrete coefficient operator}
\label{discrete-problems}
The purpose of this final section is to show discrete variants 
of Theorem~\ref{main-cont} for various 
finite element discretizations of the problem \eqref{stokes1}.
More precisely, we will consider finite element discretizations of the
system \eqref{stokes1} of the form:

Find $(u_h,p_h) \in V_h\x Q_h$ such that
\begin{equation}\label{stokes1-h}
\begin{array}{rll}
\<u_h , v \> + \eps^2 \<Du_h, Dv \> &+\, \<p_h,\div v \> &= \< f,v\>, \quad v \in
V_h,\\
\<\div u_h, q \> &   &= \< g, q \>, \quad q \in Q_h.
\end{array}
\end{equation}
Here $V_h$ and $Q_h$ are finite element spaces such that $V_h \x Q_h
\subset H_0^1(\Omega;\R^n) \x L^2_0(\Omega)$, and $h$ is the
discretization parameter. Alternatively, these
problems can be written on the form 
\[
\A_{\eps,h}\begin{pmatrix} u_h \\ p_h \end{pmatrix}
= \begin{pmatrix} f_h \\ g_h \end{pmatrix},
\]
where the coefficient operator $\A_{\eps,h}$ is acting on
elements of $V_h \x Q_h$. 
In the examples below we will, for simplicity, only consider
discretizations where the finite element space $V_h \times Q_h \subset
X_{\eps}$ for all $\eps$ in the closed interval $[0,1]$. This implies
that also the pressure space $Q_h$ is a subspace of $H^1$.
The proper discrete uniform 
inf--sup conditions we shall establish will be of the form
\begin{equation}\label{u-inf-sup-h}
\sup_{v \in V_h}\frac{\<\div v, q\>}{\| v \|_{L^2\cap
    \eps H^1}} \ge \alpha
\| q \|_{H^1 + \eps^{-1}L^2,h}, \quad q \in Q_h,
\end{equation}
where the positive constant $\alpha$ is independent of both $\eps$ and
$h$. Here the discrete norm $\|\cdot \|_{H^1 + \eps^{-1} L^2,h}$ is
defined as
\[
\|q \|_{H^1 + \eps^{-1} L^2,h}^2 = \inf_{\substack{q_1,q_2 \in
    Q_h\\q = q_1 + q_2}} 
(\|q_1 \|_{H^1}^2 + \eps^{-2}\|q_2 \|_{L^2}^2 ), \quad q \in Q_h.
\]
The technique we will use to establish the discrete inf--sup condition
\eqref{u-inf-sup-h}
is in principle rather standard. We will just rely on the
corresponding continuous condition \eqref{u-inf-sup} and a bounded
projection operator into the velocity space $V_h$. The key property
is that the projection operator $\Pi_h$ commutes properly with the
divergence operator, cf. \eqref{commute} below, and that it is
uniformly bounded
in the proper operator norm. Such projection operators are frequently
referred to as Fortin operators.

We
will restrict the discussion below to two key examples, the Mini
element and the Taylor--Hood element. For both these examples we will
construct interpolation operators $\Pi_h : L^2(\Omega;\R^n) \to V_h$
which are uniformly bounded, with respect to $h$,  in both $L^2$ and
$H_0^1$. Therefore,
these operators will be uniformly bounded operators in $\LL((L^2 \cap \eps
H_0^1)(\Omega;\R^n))$, i.e., we have 
\begin{equation}\label{bounded}
\| \Pi_h \|_{\LL((L^2 \cap \eps H_0^1)} \quad \text{is bounded
  independently of } \eps \, \text{and } h.
\end{equation}
It is easy to see that this is equivalent  to the requirement that 
$\Pi_h$ is uniformly bounded with respect to $h$ in $\LL(L^2)$ and $\LL(H_0^1)$.
Furthermore, the operators $\Pi_h$ will satisfy a commuting
relation of the form
\begin{equation}\label{commute}
\<\div \Pi_h v, q \> = \< \div v, q \> \quad v \in H_0^1(\Omega;\R^n), \,
q \in Q_h.
\end{equation}
As a consequence, the discrete uniform inf--sup condition \eqref{u-inf-sup-h}
will follow from the corresponding condition \eqref{u-inf-sup} in the
continuous case. We recall from Theorem~\ref{main-cont} that
\eqref{u-inf-sup}
holds for any bounded Lipschitz domain, and without any 
convexity assumption, and as a consequence of the analysis below the 
discrete condition \eqref{u-inf-sup-h} will also hold without any
convexity
assumptions. The key ingredient in the analysis below is the
construction of a uniformly bounded interpolation operator $\Pi_h$.
In the case of the Mini element the construction we will present is
rather standard, and resembles the presentation already done in 
\cite{a-b-f},  where  this element was originally proposed (cf. also 
\cite[Chapter VI]{BrezziFortin}). However, for the Taylor--Hood
element
the direct construction of a bounded, commuting interpolation operator is not
obvious. In fact, 
most of the stability proofs found in the literature for
this discretization typically
uses an alternative approach, cf. for example \cite[Section
VI.6]{BrezziFortin}
and the discussion given in the introduction of \cite{falk}.
An exception is \cite{falk}, where a projection
satisfying \eqref{commute} is constructed. However, 
this operator is not bounded in $L^2$.
Below we propose 
a new construction of projection operators satisfying \eqref{commute}
by utilizing a 
technique for the Taylor--Hood method which is similar to the 
construction for Mini element presented below. The new projection
operator
will be bounded in both $L^2$ and $H^1$, and hence it satisfies \eqref{bounded}.
This analysis is restricted to quasi--uniform meshes in two space dimensions.

In the present case, the discrete inf--sup condition \eqref{u-inf-sup-h}
will imply uniform stability of the discretization in the proper 
norms introduced above. 
As a consequence, we are able to derive preconditioners $\B_{\eps,h}$,
such that the condition numbers of the corresponding operators 
$\B_{\eps,h}\A_{\eps,h}$ are bounded uniformly with respect to the
perturbation parameter $\eps$ and the discretization parameter $h$.
The operator $\B_{\eps,h}$ can be taken as a block diagonal
operator of the form \eqref{u-preconditioner}, but where the elliptic
operators are replaced by the corresponding discrete
analogs.
In fact, to obtain an efficient preconditioner the inverses of the
elliptic operators which appear should be replaced 
by corresponding elliptic preconditioners, constructed for example by a
standard multigrid procedure. We refer to \cite{review}, see in
particular
Section 5 of that paper, 
for a discussion on the relation between stability estimates and the
construction of uniform preconditioners. In particular, the results
for the Taylor--Hood method presented below explains the uniform
behavior
of the preconditioner $\B_{\eps,h}$ observed in the numerical
experiment
reported in Table~\ref{table1} above.  

\subsection{The discrete inf--sup condition}
The rest of the paper is devoted to the construction of proper
interpolation operators $\Pi_h$ for the Mini element and the
Taylor--Hood element, i.e, we will construct interpolation operators
$\Pi_h : L^2(\Omega;\R^n) \to V_h$ such that \eqref{bounded} and
\eqref{commute} holds. 
We will assume that the domain $\Omega$ is a polyhedral domain
which is triangulated  by a family of shape regular, simplicial meshes
$\{\T_h \}$
indexed by
decreasing values of the mesh parameter $h = \max_{T\in \T_h} h_T$.
Here $h_T$ is the diameter of the simplex $T$. We recall that the mesh is shape regular
if the there exist a positive constant $\gamma_0$ such that for all values of the mesh parameter $h$
\[
h_T^n \le \gamma_0 |T|, \quad T \in T_h.
\]
Here $|T|$ denotes the volume of $T$.

\subsubsection{The Mini element}
We recall that for this element the velocity space, $V_h$, consists of 
linear combinations of continuous piecewise linear vector fields
and local bubbles. More precisely, $v \in V_h$
if and only if
\[
v = v^1 + \sum_{T \in \T_h}c_T b_T,
\]
where $v^1$ is a continuous piecewise linear vector field, $c_T \in \R^n$,
and $b_T \in \P_{n+1}(T)$ is the bubble function
with respect to $T$, i.e. the unique polynomial of degree $n+1$ which 
vanish on $\partial T$ and with $\int_T b_T\, dx = 1$.
The pressure space $Q_h$ is the standard space of continuous piecewise linear 
scalar fields.

In order to define the operator $\Pi_h$ we will utilize the fact that the
space $V_h$ can be decomposed into two subspaces, $V_h^b$,
consisting of all functions which are identical to zero
on all element boundaries, i.e. $V_h^b$ is the span of the bubble functions, 
and $V_h^1$
consisting of continuous piecewise linear vector fields.
Let $\Pi_h^b : L^2(\Omega;\R^n) \to V_h^b$ be defined by,
\[
\< \Pi^b_h v , z \> = \<v , z \> \qquad \forall z \in Z_h,
\]
where $Z_h$ denotes the space of piecewise constants vector fields.
Clearly this uniquely determines $\Pi^b_h$. Furthermore, a scaling
argument,
utilizing equivalence of norms, shows that the local operators $\Pi^b_h$
are uniformly bounded, with respect to $h$, in $L^2(\Omega;\R^n)$.

The operator $\Pi^b_h$ will satisfy property \eqref{commute}
since for all $v \in H^1_0(\Omega;\R^n)$ and $q \in Q_h$, we have
\begin{equation}\label{commute-bubble}
\<\div \Pi_h^b v,q\> = - \<\Pi_h^bv, \grad q\> = - \<v, \grad q \>
= \< \div v, q \>,
\end{equation}
where we have used that $\grad Q_h \subset Z_h$. 

The desired operator $\Pi_h$
will be of the form
\[
\Pi_h = \Pi_h^b(I - R_h) + R_h,
\]
where $R_h : L^2(\Omega;\R^n) \to V_h^1$ will be specified below.
Note that
\[
I- \Pi_h =  (I - \Pi_h^b)(I- R_h),
\]
and therefore
\[
\< \div (I-\Pi_h) v,q \> = \< \div (I-\Pi^b_h)(I- R_h) v,q \>)= 0
\]
for all $q \in Q_h$. Hence,
the operator $\Pi_h$ satisfies \eqref{commute}.

We will take $R_h$ to be the Clement interpolant onto piecewise linear
vector fields, cf. \cite{clement}. Hence, in particular, the operator $R_h$
is local, it preserves constants, and it is stable in $L^2$ and
$H_0^1$. More precisely, we have for any $T \in \T_h$ that
\begin{equation}\label{int-prop}
\| (I - R_h) v \|_{H^j(T)} \le c h_{\Omega_T}^{k-j} \|v \|_{H^k(\Omega_T)}, \qquad 0\le j \le k
\le 1,
\end{equation}
where the constant $c$ is independent of $h$ and $v$. 
Here $\Omega_T$ denote the macroelement consisting of $T$ and all
elements
$T' \in \T_h$ such that $T \cap T' \neq \emptyset$, and 
$h_{\Omega_T} = \max_{T\in \T_h, T \subset \Omega_h} h_T$.
It also follows from the shape regularity of the family
$\{\T_h \}$ that the covering $\{\Omega_T \}_{T \in \T_h}$ has a
bounded overlap. Therefore, it follows from \eqref{int-prop}
and the $L^2$ boundedness of $\Pi_h^b$ that $\Pi_h$ is uniformly
bounded 
in $\LL(L^2(\Omega;\R^n))$.
Furthermore, by combining \eqref{int-prop}
with a standard inverse estimate for polynomials we have for any $T
\in \T_h$ that 
\begin{align*}
\|\Pi_h v \|_{H^1(T)} &\le \| \Pi_h^b (I - R_h)v \|_{H^1(T)} + \| R_h v \|_{H^1(T)}\\
&\le c (h_T^{-1}\| \Pi_h^b (I - R_h)v \|_{L^2(T)} + \|  v \|_{H^1(T)}) \\
&\le c (h_T^{-1}\| (I - R_h)v \|_{L^2(T)} + \|  v \|_{H^1(T)} )\\
&\le c (h_T^{-1} h_{\Omega_T} + 1)\| v \|_{H^1(\Omega_T)}\\
&\le c \| v \|_{H^1(\Omega_T)},
\end{align*}
where we have used that $h_T^{-1} h_{\Omega_T}$ is uniformly bounded
by shape regularity.
This implies that $\Pi_h$ is uniformly bounded in $\LL(H_0^1(\Omega;\R^n))$.
We have therefore verified \eqref{bounded}. Together with
\eqref{commute} this implies \eqref{u-inf-sup-h}.  
In Table~\ref{table2} below we present results for the Mini element
which are completely parallel to results for the Taylor--Hood element 
presented in Table~\ref{table1} above. As we can see, 
by comparing the results of the two tables, the effect of
the different discretizations seems to minor, as long as the mesh is
the same.

\begin{table}
\begin{center}
\begin{tabular}{|c|c||c|c|c|c|} \hline 
domain  & $\eps \backslash h$  &   $2^{-2}$  & $2^{-3}$ & $2^{-4}$ & $2^{-5}$ \\  \hline  
\multirow{3}{20mm}{$\quad\Omega_1$} 
      & 1 & 26.6 & 24.7 & 25.7 & 27.0 \\ 
      & 0.1 & 9.5 & 12.7 & 15.1 & 17.0 \\ 
      & 0.01 & 3.4 & 4.0 & 5.7 & 9.2 \\  \hline
\multirow{3}{20mm}{$\quad\Omega_2$} 
      & 1 & 27.0 & 20.8 & 19.2 & 18.6 \\ 
      & 0.1 & 9.0 & 12.3 & 15.1 & 16.7  \\ 
      & 0.01 & 3.4 & 4.0 & 5.5 & 8.3 \\ \hline 
\multirow{3}{20mm}{$\quad\Omega_3$} 
      & 1 & 15.8 & 17.3 & 17.8 & 17.9 \\ 
      & 0.1 & 8.8 & 12.4 & 15.1 & 16.7  \\ 
      & 0.01 & 3.4 & 4.0 & 5.5 & 8.3 \\ \hline 
\end{tabular}
\caption{Condition numbers for the operators $\B_{\eps,h}\A_{\eps,h}$ discretized with the Mini element.} 
\label{table2}
\end{center}
\end{table}

\subsubsection{The  Taylor--Hood element}
Next we will consider the classical Taylor--Hood element.
We will restrict the discussion to two space dimensions, and we will
assume that the family of meshes $\{ \T_h \}$ is quasi--uniform.
More precisely, we assume that there is a mesh independent constant
$\gamma_1 > 0$ such that 
\begin{equation}\label{quasi-uniform}
h_T \ge \gamma_1 \, h, \quad T \in \T_h,
\end{equation}
where we recall that $h = \max_T h_T$.
For the Taylor--Hood element the velocity space, $V_h$, consists of 
continuous piecewise quadratic vector fields,
and as for the Mini element above $Q_h$ is the 
standard space of continuous piecewise linear 
scalar fields. Note that if we have established 
the discrete inf--sup condition for a pair of spaces $(V_h^-, Q_h)$,
where $V_h^-$ is a subspace of $V_h$, then this condition will also 
hold for the pair $(V_h,Q_h)$. This observation will be utilized here.

For technical reasons we will assume in the rest of this section that
\emph{any 
$T \in \T_h$ has at most 
one edge in $\partial \Omega$.} Such an assumption is frequently made
for convenience when the Taylor--Hood element is analyzed, cf. for example \cite[Proposition
6.1]{BrezziFortin}, since most approaches requires a special
construction near the boundary. On the other hand, this assumption
will not hold for many simple triangulations. Therefore, in
Section~\ref{remove} below we will refine our analysis, and, as a
consequence, this 
assumption will be relaxed.

We let $V_h^- \subset V_h \subset H_0^1(\Omega;\R^2)$ be the space of piecewise quadratic vector
fields which has the property that on each edge of the mesh the normal
components
of elements in $V_h^-$ are linear. Each function $v$ in the space $V_h^-$ 
can be determined from its values at each interior vertex of the mesh, and of the
mean value of the tangential component along each interior edge.
In fact, in analogy with the discussion of the Mini element above,
the space $V_h^-$ can be decomposed as $V_h^- = V_h^1 \oplus V_h^b$.
As above the space $V_h^1$ is the space of continuous piecewise linear
vector fields, while the space $V_h^b$ in this case is spanned by
quadratic 
``edge bubbles.'' To define this space of bubbles we let
\[
\Delta_1(\T_h) =\Delta_1^i(\T_h)\cup \Delta_1^{\partial}(\T_h)
\]  
be the set of the edges of the mesh
$\T_h$, where $\Delta_1^i(\T_h)$ are the interior edges and 
$\Delta_1^{\partial}(\T_h)$ are the edges on the boundary of $\Omega$.
Furthermore, if $T \in \T_h$ then $\Delta_1(T)$ are the set of edges of $T$,
and $ \Delta_1^i(T)= \Delta_1(T)\cap \Delta_1^i(\T)$.

For each $e \in \Delta_1(\T_h)$ 
we let $\Omega_e$ be the associated macroelement consisting of the
union of all $T \in \T_h$ with $e \in \Delta_1(T)$.
The scalar function $b_e$ is the unique continuous and piecewise quadratic function on $\Omega_e$
which vanish on the boundary of $\Omega_e$, and with $\int_{e}
b_e \, ds = |e|$, where $|e|$ denotes the length of $e$.
The space $V_h^b$ is defined as
\[
V_h^b = \sp \{b_et_e \, | \, e \in \Delta_1^i(\T_h) \, \},
\]
where $t_e$ is a tangent vector along $e$ with length $|e|$. 
Alternatively, if $x_i$ and $x_j$ are the vertices corresponding to
the endpoints of $e$ then the vector field $\psi_e = b_et_e$ is
determined up to a sign as
$\psi_e = 6\lambda_i\lambda_j(x_j -x_i)$, where $\{\lambda_i \}$ are
the piecewise linear functions corresponding to the barycentric
coordinates,
i.e., $\lambda_i(x_k) = \delta_{i,k}$ for all vertices $x_k$. In
particular,
\[
\int_e \psi_e \cdot (x_j - x_i) \, ds = 6\int_e \lambda_i\lambda_j \,
ds |e|^2 = |e|^3.
\]
As above the desired interpolation operator $\Pi_h$
will be of the form
\[
\Pi_h = \Pi_h^b(I - R_h) + R_h,
\]
where $R_h : L^2(\Omega;\R^3) \to V_h^1$  is the same Clement operator
as above, and where $\Pi_h^b: L^2(\Omega;\R^n) \to V_h^b$ 
needs to be specified. In fact, to perform a construction similar to
the one we did for the Mini element it will be sufficient 
to construct $\Pi_h^b$ such that it is $L^2$--stable, and satisfies the
commuting relation \eqref{commute-bubble}.

We will need to separate the triangles which have an edge on the
boundary of $\Omega$ from the interior triangles. With this purpose  we define
\[
\T_h^{\partial} = \{ T \in \T_h \, | \, T \cap \partial \Omega \in
\Delta_1^{\partial}(\T_h) \, \} \quad \text{and} \quad \T_h^i = \T_h
\setminus \T_h^{\partial}.
\]  
In order to define the operator $\Pi_h^b$ 
we introduce $Z_h$ as the lowest order Nedelec space with respect to
the mesh $\T_h$. Hence, if $z \in Z_h$ then on any $T \in \T_h$,
$z$ is a linear  vector field such that $z(x) \cdot x$ is also
linear. Furthermore, for each $e \in \Delta_{1}(\T_h)$ the
tangential component of $z$ is continuous. As a consequence,
$Z_h \subset H(\curl;\Omega)$ where the operator $\curl$ denotes the 
two dimensional analog of the curl--operator given by 
\[
\curl z = \curl (z_1,z_2) = \partial_{x_2}z_1 - \partial_{x_1}z_2.
\]
It is well known that 
the proper degrees of freedom for the space $Z_h$ is the mean value 
of the tangential components of $v$, $v \cdot t$,  with respect to
each edge in $\Delta_1(\T_h)$.
Furthermore, we let
\[
Z_h^0 = \{ z \in Z_h \, | \, \int_{\partial T} z\cdot t \, ds = 0, \,
T \in \T_h^{\partial} \, \}.
\]
Alternatively, the elements of $Z_h^0$ are those vector fields in
$Z_h$
with the property that $\curl z|_{T} = 0$ if $T \in \T_h^{\partial}$, i.e.,
$z$ is a constant vector field on $T$ for $T \in \T_h^{\partial}$.
It is a key observation that the mesh assumption given above, that any
$T \in \T_h$ intersects $\partial \Omega$ in at most one edge,
implies that the spaces $V_h^b$ and $Z_h^0$ have the
same dimension. Furthermore, we note that $\grad q \in Z_h^0$ for any
$q \in Q_h$. 

For each $e \in \Delta_1(\T_h)$ let $\phi_e \in Z_h$ be the basis function
corresponding to the Whitney form, i.e., $\phi_e$ satisfies 
\[
\int_e (\phi_e \cdot t_e) \, ds = |e|, \quad \text{and } \int_{e'}
(\phi_e \cdot t_{e'})\, ds = 0, \quad e' \neq e,
\]
where, as above, $t_e$ is a tangent vector of length $e$. Hence, if
$e = (x_i,x_j)$ then the vector field $\phi_e$ can be expressed in
barycentric coordinates as
\[
\phi_{e} = \lambda_i \grad \lambda_j -\lambda_j \grad \lambda_i.
\] 
Any $z \in Z_h^0$ can be written uniquely on the form 
\[
z = \sum_{e \in \Delta_1^i(\T_h)} a_e \phi_e + \sum_{e \in
  \Delta_1^{\partial}(\T_h)} c_e \phi_e,
\]
where the coefficients $a_e$ corresponding to interior edges can be
chosen arbitrarily, but  where the coefficients $c_e$ for each
boundary edge should be chosen such that $\curl z = 0$ on the 
associated triangle in $\T_h^{\partial}$.
We note that there is a natural mapping $\Phi_h$ between the spaces $V_h^b$ and
$Z_h^0$  
given by $\Phi_h(\psi_e) = \phi_e$ for all interior edges, or alternatively,
\begin{equation}\label{edgeDOF}
\int_e \Phi_h(v)\cdot t_e \, ds = |e|^{-2}\, \int_e v \cdot t_e \, ds, \quad e \in \Delta_1^i(\T_h).
\end{equation}
Below, we will use $V_h^b(T)$ and $Z_h^0(T)$ to denote the restriction of the spaces $V_h^b$ and $Z_h^0$
to a single element $T$, and $\Phi_T$ 
will denote the corresponding restriction of the map $\Phi_h$.

We will define the operator $\Pi_h^b : L^2(\Omega;\R^n) \to V_h^b$ by,
\begin{equation}\label{pi-b-def}
\< \Pi^b_h u , z \> = \< u , z \> \qquad \forall z \in Z_h^0.
\end{equation}
To show that this operator is well--defined the following general formula
for integration of 
products of barycentric coordinates
over a triangle $T$ will be useful
(cf. for example \cite[Section 2.13]{lay-schu})
\begin{equation}\label{bary-int}
\int_T\lambda_1^{\alpha_1}\lambda_2^{\alpha_2}\lambda_3^{\alpha_3} \,
dx 
= \frac{2 \alpha!}{(2 + |\alpha|)!}|T|,
\end{equation}
where 
$\alpha! = \alpha_1!\alpha_2!\alpha_3!$, $|\alpha| = \sum_i \alpha_i$
and $|T|$ is the area of $T$.

\begin{lem}\label{triangle1}
Let $T \in \T_h^i$ with edges $e_1,e_2,e_3$. For any $v \in V_h^b(T)$ we have 
\[
\int_T v \cdot \Phi_T(v) \, dx \ge \frac{1}{5}|a|^2|T|,
\]
where $\Phi_T = \Phi_h|_T$, $v= \sum_i a_i\psi_{e_i}$ and $|a|^2 = \sum_i a_i^2$.
\end{lem}

\begin{proof}
A direct computation gives
\[
\int_T v \cdot \Phi_T(v) \, dx =  
\sum_{i=1}^3\sum_{j=1}^3 a_i a_j \,
\int_T \psi_{e_i} \cdot\phi_{e_j} \, dx =
\sum_{i=1}^3\sum_{j=1}^3 a_i a_j \,
\int_T b_{e_i} \phi_{e_j} \cdot t_{e_i}\, dx = a^T M a,
\]
where the $3 \times 3$ matrix $M$ is given by
\[
M = \{M_{i,j} \}_{i,j =1}^3 = \{\int_T \psi_i \cdot  \phi_{e_j} \, dx \}_{i,j =1}^3
\{\int_T b_{e_i} \phi_{e_j} \cdot t_{e_i}\, dx \}_{i,j =1}^3,
\]
The desired result will follow from the diagonal dominance of this matrix.

Let $x_j$ be the vertex opposite $e_j$, and let $\lambda_j$ be the 
corresponding barycentric coordinate on $T$.
Then the first diagonal element $M_{1,1}$ of the matrix $M$ is given by
\[
M_{1,1} = 6\, \int_T \lambda_2 \lambda_3 (\lambda_2 \grad \lambda_3  - \lambda_3 \grad \lambda_2)(x_3 - x_2) \, dx
= 6\,\int_T  (\lambda_2^2 \lambda_3 + \lambda_2 \lambda_3^2) \, dx = 2|T|/5,
\]
where we have used formula \eqref{bary-int} in the final step.
Actually, from 
this formula we derive that all the diagonal elements are given by 
$M_{i,i} = 2|T|/5$, and
similar calculations for the off--diagonal elements gives
$|M_{i,j}|= |T|/10$. In addition, the matrix $M$ is symmetric. 
The matrix $M$ is therefore strictly
diagonally dominant, and by the Gershgorin circle theorem all
eigenvalues are bounded below by $|T|/5$. By combining this with the
fact that $M$ is symmetric we conclude that 
$a^TMa \ge |a|^2|T|/5$, and this is the desired bound.
\end{proof}

The next lemma  is a variant of the result above for $T \in \T_h^{\partial}$.
\begin{lem}\label{triangle2}
Let $T \in \T_h^\partial$ with interior edges $e_1,e_2$, and where $e_3$ is the edge on the boundary.
For any $v \in V_h^b(T)$ we have 
\[
\int_T v \cdot \Phi_T(v) \, dx \ge  \frac{1}{2}|a|^2|T|,
\]
where $v =a_1\psi_{e_1}  + a_2\psi_{e_2}$ and $|a|^2 = a_1^2 + a_2^2$.
\end{lem}

\begin{proof}
Let $v =a_1\psi_{e_1}  + a_2\psi_{e_2} = 6a_1\lambda_2\lambda_3(x_3 - x_2) + 6a_2\lambda_1\lambda_3(x_3 - x_1)$.
It is a key observation that in this case $\Phi_T(v)$ is simply given
as
$\Phi(v) = -a_1\grad \lambda_2  - a_2\grad \lambda_1$. 
Since $\grad \lambda_i \cdot t_{e_i} \equiv 0$ we therefore obtain
from \eqref{bary-int} that
\begin{align*}
\int_T v \cdot \Phi_T(v) \, dx &= 6a_1^2 \int_T \lambda_2\lambda_3 \grad \lambda_2
\cdot (x_2 - x_3) \, dx
+ 6a_2^2 \int_T \lambda_1\lambda_3 \grad \lambda_1 \cdot (x_1-x_3) \, dx\\
&= 6 a_1^2\int_T \lambda_2 \lambda_3 \, dx 
+ 6 a_2^2 \int_T \lambda_1 \lambda_3 \, dx 
= \frac{1}{2}|a|^2|T|.
\end{align*}
This completes the proof.
\end{proof}

\begin{lem}\label{inf-sup-int}
There is a positive constant $c_0$, independent of $h$, such that for each $v
\in V_h^b $
\[
\sup_{z \in Z_h^0} \frac{\< v,z \>}{\| z \|_{L^2(\Omega)}} \ge c_0 \| v
\|_{L^2(\Omega)}.
\]
\end{lem}

\begin{proof}
Let $v \in V_h^b$ be given, i.e., 
$v = \sum_{e \in \Delta_1^i(\T_h)} a_e \psi_e$.
We simply choose the 
corresponding $z = \Phi_h(v) = \sum_{e \in \Delta_1^i(\T_h)}a_e
\phi_e + \sum_{e \in \Delta_1^\partial(\T_h)} c_e \phi_e\in Z_h^0$. 
It follows from scaling and shape regularity that the two norms of 
$z$, given by 
\begin{equation}\label{norm-z}
\| z \|_{L^2(\Omega)} \quad \text{and} \quad  (\sum_{T \in
  \T_h}\sum_{e \in \Delta_1^i(T)}a_e^2)^{1/2}
\end{equation}
are equivalent uniformly in $h$.
Correspondingly, the two norms
\begin{equation}\label{norm-v}
\| v \|_{L^2(\Omega)} \quad \text{and} \quad  (\sum_{T \in
  \T_h}|T|^2\sum_{e \in \Delta_1^i(T)}a_e^2)^{1/2}
\end{equation}
are uniformly equivalent. As a consequence of 
these properties, combined with Lemmas \ref{triangle1} and \ref{triangle2},
we obtain
\[
\< v,\Phi_h(v) \> = \sum_{T \in \T_h} \int_T  \Phi_T(v)\cdot v\, dx
\ge \frac{1}{5} \sum_{T \in \T_h} |T| \sum_{e \in \Delta_1^i(T)}a_e^2
\ge c_0  \| \Phi_h(v) \|_{L^2(\Omega)} \, \| v \|_{L^2(\Omega)}.
\]
where $c_0>0$ is independent of $h$.
This completes the proof.
\end{proof}

It is a direct consequence of Lemma~\ref{inf-sup-int} that the
operators $\Pi_h^b$ are uniformly bounded in $\LL(L^2(\Omega; \R^2))$. In
fact, the associated operator norm is bounded by $c_0^{-1}$. 
Note that in contrast to the situation for the Mini element,
the operator $\Pi_h^b$ is not local in this case. However,
if the mesh is quasi--uniform we obtain from \eqref{quasi-uniform}
that
\begin{equation}\label{q-uniform-est}
\| \Pi_h^b u \|_{H^1(\Omega)} \le c \big( \sum_{T \in \T_h} h_T^{-2}\|
\Pi_h^b u \|_{L^2(T)}^2 \big)^{1/2} \le c \gamma_1^{-1}h^{-1}\|
\Pi_h^bu\|_{L^2(\Omega)}.
\end{equation}
As a further consequence we now obtain.

\begin{thm}\label{main}
The operator $\Pi_h$ satisfies properties \eqref{bounded}
and \eqref{commute}.
\end{thm}

\begin{proof}
To show that the operator $\Pi_h$ fulfills the condition 
\eqref{commute}, it is enough to show that the operator $\Pi_h^b$
satisfies the corresponding condition \eqref{commute-bubble}. However, 
since $\grad Q_h \subset Z_h^0$
this follows exactly as before, since
\[
\< \div \Pi_h^b v,q\> = - \< \Pi_h^bv, \grad q\> = - \< v, \grad q \>
= \< \div v, q \>.
\]
Furthermore, to show \eqref{bounded} it is enough to show that $\Pi_h$
is uniformly bounded with respect to $h$ in both $\LL(L^2(\Omega; \R^2))$
and $\LL(H_0^1(\Omega;\R^2))$. However, the $L^2$--result follows from  the
corresponding bounds
for the operators $\Pi_h^b$ and $R_h$. Finally, by combining \eqref{int-prop},
\eqref{q-uniform-est} and the boundedness of $\Pi_h^b$ in $L^2$ we obtain 
\begin{align*}
\|\Pi_h v \|_{H^1(\Omega)} &\le \| \Pi_h^b (I - R_h)v \|_{H^1(\Omega)}
+ \| R_h v \|_{H^1(\Omega)}\\
&\le c (h^{-1}\| \Pi_h^b (I - R_h)v \|_{L^2(\Omega)} + \|  v
\|_{H^1(\Omega)}) \\
&\le c (c_0^{-1}h^{-1}\| (I - R_h)v \|_{L^2(\Omega)} + \|  v \|_{H^1(\Omega)}) \\
&\le c \| v \|_{H^1(\Omega)},
\end{align*}
and this is the desired uniform bound in $\LL(H_0^1(\Omega;\R^2))$.
\end{proof}

\subsubsection{More general triangulations}\label{remove}
The analysis of the Taylor--Hood method given above leans havily on
the assumption that there are no triangles in $\T_h$ with more than
one edge on the boundary of $\Omega$. This assumption simplifies the
analysis, but it is not necessary. The purpose of this section 
is to relax this assumption.

We let $\T_h^{\partial,1}$ and  $\T_h^{\partial,2}$ denote the subset
of triangles in $\T_h$ with one or two edges in $\partial \Omega$,
respectively. We let $\T_h^{\partial} = \T_h^{\partial,1}\cup
\T_h^{\partial,2}$ 
be the set of all boundary triangles, and as before
$\T_h^i = \T_h \setminus \T_h^{\partial}$.
We note that $T \in \T_h^{\partial,2}$ then
there is a unique associated triangle $T^- \in \T_h$ such that $T \cap T^- \in
\Delta_1^i(\T_h)$. We will denote this interior edge associated any $T
\in \T_h^{\partial,2}$ by $e_T$, and we will use
$T^*$ to denote the macroelement defined by the two
triangles $T$ and $T^-$.
The set of all interior edges of the form $e_T, \, T \in
\T_h^{\partial,2}$, 
will be denoted $\Delta_1^{i,2}(\T_h)$, while 
$\Delta_1^{i,1}(\T_h) = \Delta_1^{i}(\T_h) \setminus \Delta_1^{i,2}(\T_h)$.
Throughout this section \emph{we will assume that} all the triangles
of the form $T^-$ are interior triangles, i.e., 
\[
T^- \in \T_h^i \quad \text{for all } T \in \T_h^{\partial,2}. 
\]
The interpolation operators $\Pi_h$ and $\Pi_h^b$ will be defined as
above, with the only exception that the definition of the bubble space 
$V_h^b$ is changed slightly in the neighborhood of the edges in $\Delta_1^{i,2}(\T_h)$.
We note that the result of Lemma~\ref{triangle2}
still holds if $T \in \T_h^{\partial,1}$.
To establish the result of Lemma~\ref{inf-sup-int}, and as a
consequence Theorem~\ref{main}, in the present case we basically
need an anolog of Lemma~\ref{triangle2} for triangles $T \in
\T_h^{\partial,2}$. More precisely, we need to modify the definition 
of the map $\Phi_h$,  used in proof of Lemma~\ref{inf-sup-int}, 
to the present case. Actually, the map $\Phi_h$
will not be defined locally on the triangles $T \in \T_h^{\partial,2}$,
but rather on the corresponding macroelement $T^*$.
As in the previous section the space $Z_h^0$ is taken to be the
subspace
of the Nedelec space $Z_h$ such that $z \in Z_h^0$ is constant on the
triangles in $\T_h^{\partial}$. To be able to define the interpolation
operator $\Pi_h^b$, mapping into the bubble space, by \eqref{pi-b-def},
the two spaces $V_h^b$ and $Z_h^0$ must be balanced. In particular,
they should have the same dimension. 
However, in the present case the
dimension of the space $Z_h^0$ is not the same as the number of interior
edges. To see this just consider the restriction $Z_h^0(T^*)$ of
$Z_h^0$ to a macroelement macroelement $T^*$,
cf. Figure~\ref{macroelement} below.
The dimension of the space $Z_h^0(T^*)$ is four,
while there are only three interior edges, namely the three edges of
$T^-$. To compensate for this we will extend the space of bubble
functions $V_h^b$, by including also ``normal bubbles'' on
the edges $e_T,\, T \in \T_h^{\partial,2}$.

In the present case we define the space $V_h^b \subset V_h$ by
\[
V_h^b = \sp (\{\, b_et_e\, |\,  e \in \Delta_1^{i}(\T_h)\, \} \cup \{\, b_e n_e \,|\, e
\in \Delta_1^{i,2}(\T_h)\, \}).
\]
Here $t_e$ and $n_e$ are tangent and normal vectors to the edge $e$
with length $|e|$. With this definition the space  $V_h^b$ has the
same dimension as $Z_h^0$. Furthermore, the map $\Phi_h :V_h^b \to Z_h^0$
will be defined to satisfy property \eqref{edgeDOF} for all edges in
$\Delta_1^{i,1}(\T_h)$.
Note that this specifies $\Phi_h$ on all triangles in $\T_h$,
except for the ones that belongs to the macroelements $T^*$, $T \in
\T_h^{\partial,2}$.
To complete the definition of $\Phi_h$ we need to specify its
restriction to each macroelement $T^*$.

Consider a macroelement $T^*$ of the form given in Figure~\ref{macroelement}.
\begin{figure}
\begin{center}
\scalebox{0.4}{
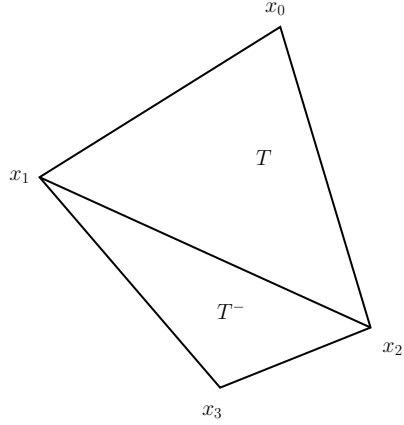}
\end{center}
\caption{The macroelement $T^*$.}
\label{macroelement}
\end{figure}
Here the edge $e_T$ has endpoints denoted by $x_1$ and $x_2$,
the third boundary vertex of $T^*$ is 
$x_0$, while the single interior vertex of $T^*$ is $x_3$.
We will use $e_i$ to denote the edge opposite $x_i, \, i=1,2$, of the triangle $T$,
while $e_i^-$ are the corresponding edges of the triangle $T^-$.
We let $\Phi_T : V_h^b(T^*) \to Z_h^0(T^*)$ be the restriction of
$\Phi_h$
to $T^*$. To be compatible with the definition of $\Phi_h$ outside the
macroelements $T^*$ the map $\Phi_T$ has to satisfy 
condition \eqref{edgeDOF} on the edges $e_i^-$, i.e., 
\begin{equation}\label{compatible}
\int_{e_i^-}\Phi_T(v) \cdot t_{e_i^-}\, ds = |e_i^-|^{-2} \, \int_{e_i^-}v
\cdot t_{e_i^-}\, ds, \quad i=1,2,
\end{equation}
where $t_{e_i^-}$ is a vector tangential to $e_i^-$.
As a basis for the space $Z_h^0(T^*)$ we will use the
functions
\[
\phi_i^- = \lambda_i \grad \lambda_3 - \lambda_3 \grad \lambda_i,
\quad i=1,2,
\]
with support only in $T^-$, combined with the two functions $\phi_i$
given by 
\[
\phi_i = \left\{\begin{array}{cc} \grad \lambda_i  \quad &\text{on } T,\\
(-1)^i \phi_T, \quad &\text{on } T^-, \end{array} \right.
\]
for $i=1,2$, where $\phi_T = \lambda_1\grad \lambda_2 - \lambda_2\grad \lambda_1$
corresponds to the Whitney form associated the edge $e_T = (x_1,
x_2)$.
The functions $\phi_i, \phi_i^-$ for $i=1,2$ spans the space $Z_h^0(T^*)$.

We will define two basis functions 
$\psi_i$ of  $V_h^b$ as a multiple of the scalar bubble function
$b_{e_T}$,
namely, 
\begin{equation}\label{basisV_h^b}
\psi_i = b_{e_T}w_i = 6\lambda_1\lambda_2  w_i \quad
i=1,2,
\end{equation}
where the vectors $w_i$ will be chosen below.
Furthermore, the functions $\psi_i^-$, are given as
\[
\psi_i^- = 6 \lambda_i\lambda_3(x_3 - x_i) + \beta (-1)^i\lambda_1\lambda_2
(x_2 - x_1), \quad i=1,2,
\]
where $\beta = 6|T^-|/(5|T| + 4|T^-|)$.
We note that $0 < \beta < 3/2$.

The functions  $\psi_i, \psi_i^-$ for $i=1,2$ span the space
$V_h^b(T^*)$,
and we define $\Phi_T(\psi_i) = \phi_i$ and $\Phi_T(\psi_i^-) = \phi_i^-$.
A map of this form will 
satisfy the compatibility condition 
\eqref{compatible} by construction.
The motivation for the
choice of the constant $\beta$ is that we obtain 
\begin{equation}\label{orth-property}
\int_{T^*}\psi_i^- \cdot \phi_j \, dx = 0, \quad i,j = 1,2.
\end{equation}

\begin{lem} The orthogonality conditions \eqref{orth-property} hold.
\end{lem} 

\begin{proof}
The identities \eqref{orth-property} can be verified by formula \eqref{bary-int}. For
example
\begin{align*}
\int_{T^*} \psi_1^- \cdot \phi_1 \, dx &= - \beta
\int_T\lambda_1\lambda_2\grad \lambda_1 \cdot (x_2 - x_1)\, dx\\
&- \int_{T^-}(\lambda_1\grad \lambda_2 - \lambda_2\grad
\lambda_1)\cdot(6\lambda_1\lambda_3(x_3-x_1)
- \beta\lambda_1\lambda_2(x_2 - x_1)) \, dx\\
&= \beta\int_T \lambda_1\lambda_2 \, dx - \int_{T^-}(6
\lambda_1\lambda_2\lambda_3
- \beta (\lambda_1^2\lambda_2 + \lambda_1\lambda_2^2)) \, dx\\
&= (-6|T^-| + \beta(5|T| + 4|T^-|)/60 = 0.
\end{align*}
Furthermore, it is easy to check that 
\[
\int_{T^*} \psi_1^- \cdot \phi_2 \, dx = - \int_{T^*} \psi_1^- \cdot
\phi_1 \, dx,
\]
and as a consequence $\int_{T^*} \psi_1^- \cdot \phi_2 \, dx = 0$.
Similar computations can be done for the integrals involving
$\psi_2^-$.
\end{proof}

We can also verify, again using formula \eqref{bary-int}, that
the $2\times 2$ matrix $M^- = \{M^-_{i,j}\}_{i,j =1,2} = \{\int_{T^-} \psi_i^-\cdot \phi_j^- 
\}_{i,j =1,2}$ is given by 
\[
M^- = |T^-|\begin{pmatrix} (24 - \beta)/60 & (6 + \beta)/60\\ (6 +
  \beta)/60 & (24 - \beta)/60 \end{pmatrix}.
\]
For $0 < \beta < 3/2$ this symmetric matrix is strictly diagonally
dominant with both eigenvalues greater than $|T^-|/4$.

Finally, we need to investigate the 
$2\times 2$ matrix $M = \{M_{i,j}\}_{i,j =1,2} = \{\int_{T^*} \psi_i \cdot \phi_j 
\}_{i,j =1,2}$. However, first we need to define the functions $\psi_i
= \Phi_T(\phi_i)$ precisely
by specifying the vectors $w_i$ in \eqref{basisV_h^b}.
We let 
\[
\psi_1 = 6\gamma \lambda_1\lambda_2 (x_3 - x_2), \quad \text{and } 
\psi_2 = 6\gamma^{-1}\lambda_1\lambda_2 (x_3 - x_1),
\]
where the positive constant $\gamma$ will be chosen below.

Assume for a moment that $T^*$ is
a parallelogram. Then $x_3 - x_1 = x_2 - x_0$ and $x_3 - x_2 = x_1 -
x_0$,
and therefore we would have easy computable representations of the
functions $\psi_i$ on both $T$ and $T^-$. In general, we
introduce a new point $\hat x_0 \in \R^2$, depending on $T^-$, with
the property that $\hat x_0,x_1,x_2,x_3$ corresponds to the corners 
of a parallelogram, cf. Figure~\ref{macro-par}.
\begin{figure}
\begin{center}
\scalebox{0.4}{
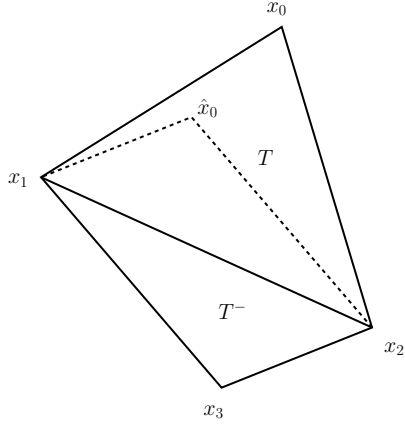}
\end{center}
\caption{The macroelement $T^*$ and the associated parallelogram.}
\label{macro-par}
\end{figure}
More precisely,
\[
\hat x_0 = x_1 - (x_3 - x_2) = x_1 + x_2 - x_3.
\]
Let $\{\hat \lambda_i\}_{i=0}^2$ be the barycentric coordinates with
respect to the
triangle $T$,
extended to linear functions on all of $\R^2$. Then $\hat \lambda_1(x_3) +
\hat \lambda_2(x_3) > 1$ and 
\[
\hat \lambda_1(\hat x_0) + \hat \lambda_2(\hat x_0)
= 2 - \hat \lambda_1(x_3) + \hat \lambda_2(x_3)
 < 1.
\]
In fact, it is a consequence of shape regularity that there is a
constant $\alpha > 0$, independent of $h$ and the choice of $T \in
\T_h^{\partial,2}$,
such that 
\begin{equation}\label{shape}
\hat \lambda_1(\hat x_0) + \hat \lambda_2(\hat x_0) \le 1 - \alpha.
\end{equation}
If we compute the matrix 
$\{\int_{T} \psi_i\cdot \phi_j \, dx\}$ we obtain 

\[
\{\int_{T} \psi_i\cdot \phi_j \, dx
\}_{i,j =1,2} = \frac{|T|}{2} \begin{pmatrix} \gamma(1 - \hat \lambda_1(\hat
  x_0)) &- \gamma\hat \lambda_2(\hat
  x_0) \\ - \gamma^{-1}\hat \lambda_1(\hat
  x_0)& \gamma^{-1}(1 - \hat \lambda_2(\hat
  x_0))\end{pmatrix}.
\]
To control the full matrix $M$ we also need to consider the 
contributions from the triangle $T^-$.
A straightforward computation, using formula \eqref{bary-int}, shows
that the matrix $\{\int_{T^-} \psi_i\cdot \phi_j \}$ is given by 
\[
\{\int_{T^-} \psi_i\cdot \phi_j \, dx
\}_{i,j =1,2}
= 
\frac{|T^-|}{5}\begin{pmatrix} \gamma & -\gamma\\ -\gamma^{-1}&
  \gamma^{-1} \end{pmatrix}.
\]
We will utilize the constant $\gamma$ to obtain a symmetric matrix
$M$.
We define 
\[
\gamma = \sqrt{\frac{2|T^-| + 5|T|\hat\lambda_1(\hat x_0)}
{2|T^-| + 5|T|\hat\lambda_2(\hat x_0)}}.
\]
This choice of $\gamma$ is motivated by the desired identity
\[
\gamma (\frac{|T|}{2}\hat\lambda_2(\hat x_0) + \frac{|T^-|}{5})
= \gamma^{-1}(\frac{|T|}{2}\hat\lambda_1(\hat x_0) + \frac{|T^-|}{5}),
\]
which can be seen to hold, and therefore the matrix $M$ is symmetric.
Furthermore, we note that 
\[
\gamma,\gamma^{-1} \le \sqrt{1 + \frac{5|T|}{2|T^-|}}.
\]
Therefore, it is a consequence of shape regularity that the positive
constant $\gamma$ is
bounded from above and below, independently of $h$ and the choice of
$T\in\T_h^{\partial,2}$.

\begin{lem}\label{M-pos}
The matrix $M$ defined above is symmetric and positive definite with 
both eigenvalues bounded below by $c_1|T|$, where $c_1 = \alpha\min(\gamma,\gamma^{-1})/2$.
\end{lem}

\begin{proof}
It follows from the calculations above that
\[
M = \begin{pmatrix} \gamma(\frac{|T|}{2}(1 - \hat \lambda_1(\hat
  x_0)) + \frac{|T^-|}{5}) &- \gamma(\frac{|T|}{2}\hat \lambda_2(\hat
  x_0)+ \frac{|T^-|}{5}) \\ - \gamma^{-1}(\frac{|T|}{2}\hat \lambda_1(\hat
  x_0)+ \frac{|T^-|}{5})& \gamma^{-1}(\frac{|T|}{2}(1 - \hat \lambda_2(\hat
  x_0)) + \frac{|T^-|}{5})\end{pmatrix}.
\]
Since $\hat\lambda_1(\hat x_0) + \hat\lambda_2(\hat x_0) \le 1 -
\alpha$
it follows from Gershgorin circle theorem that 
both eigenvalues of $M$ are
bounded below by $\frac{\alpha|T|}{2}\min(\gamma,\gamma^{-1})$.
\end{proof}

We now have the following result.

\begin{lem}\label{inf-sup-int2}
The conclusion of Lemma~\ref{inf-sup-int} holds in the present case.
\end{lem}

\begin{proof}
Let $v \in V_h^b$ be given. We first consider the situation on each macroelement $T^*$.
If $v= \sum_i (a_i^+\psi_i
+ a_i^- \psi_i^-) \in V_h^b(T^*)$ then we write $v= v^- + v^+$, where
$v^- =\sum_i a_i^- \psi_i^-$. 
Observe that Lemma~\ref{M-pos}, together with the orthogonality
property \eqref{orth-property}, implies that 
\[
\int_{T^*} v \cdot \Phi_T(v^+) \, dx = \int_{T^*} v^+ \cdot
\Phi_T(v^+) \, dx
 \ge c_1 |T||a^+|^2.
\]
Similarly, we have from the  property of the matrix $M^-$, the norm
equivalences expressed by \eqref{norm-z} and \eqref{norm-v}, and shape
regularity that 
\begin{align*}
\int_{T^-} v \cdot \Phi_T(v^-) \, dx &\ge \int_{T^-} v^- \cdot
\Phi_T(v^-) \, dx - \| v^+ \|_{L^2(T^-)}\| \Phi_T(v^-) \|_{L^2(T^-)}\\
& \ge \frac{1}{4}|T^-||a^-|^2  - c|T^-||a^-||a^+|\\
& \ge \frac{1}{8}|T^-||a^-|^2  - c_2|T||a^+|^2,
\end{align*}
where the constant $c_2$ is independent of $h$ and $T$.
By choosing $\tilde \Phi_T(v) = C \Phi(v^+) + \Phi(v^-)$, where the
constant $C$ is sufficiently large,  we can now
conclude that
\[
\int_{T^*} v \cdot \tilde \Phi_T(v) \, dx \ge c |T^*| |a|^2.
\]
We note that the map $\tilde \Phi_T$ will inherit the compatibility
condition \eqref{compatible} from the map $\Phi_T$.
By combining this result on each macroelement $T^*$,
with the map $\Phi_h$ defined previously on the rest of the triangles 
in $\T_h$, to a global map $\tilde \Phi_h$ mapping $V_h^b$, we can
conclude, as in the proof of Lemma~\ref{inf-sup-int}, that
\[
\sup_{z \in Z_h^0} \frac{\< v,z \>}{\| z \|_{L^2(\Omega)}} 
\ge \frac{\<v, \tilde \Phi_h(v)\>}{\|\tilde \Phi_h(v)\|_{L^2(\Omega)}} \ge c_0 \| v
\|_{L^2(\Omega)}.
\]
This completes the proof. 
\end{proof}
As we have noted above the result just given implies that the
conclusion of Theorem~\ref{main} holds will hold for the more general
meshes studied in this section.

\end{document}